\documentclass[10pt]{article}
\usepackage{graphicx}
\usepackage{makeidx}
\usepackage{amsmath}
\usepackage{amsthm}
\usepackage{amsfonts}
\usepackage{amssymb}
\usepackage{epsfig}
\usepackage{tabularx}
\usepackage{latexsym}
\newtheorem{lemma}{Lemma}

\def\qed{$\Box$}

\newtheorem{theorem}{Theorem}

\date{}
\begin{document}
\title{A Lower Bound on the Crossing Number of Uniform Hypergraphs}
\author{Anurag Anshu 
\footnote{National University of Singapore, Singapore. Email: a0109169@nus.edu.sg} 
\and Saswata Shannigrahi \footnote{Indian Institute of Technology Guwahati, India. Email: saswata.sh@iitg.ernet.in}
\footnote{Corresponding author}
}
\maketitle

%.................Start of Abstract.........................

\begin{abstract}
In this paper, we consider the embedding of a complete $d$-uniform geometric hypergraph with $n$ vertices in general position in $\mathbb{R}^d$, where each hyperedge is represented as a $(d-1)$-simplex, and a pair of hyperedges is defined to cross if they are vertex-disjoint and contains a common point in the relative interior of the simplices corresponding to them. As a corollary of the Van Kampen-Flores Theorem, it can be seen that such a hypergraph contains $\Omega(\frac{2^d}{\sqrt{d}})$ $n\choose 2d$ crossing pairs of hyperedges. Using Gale Transform and Ham Sandwich Theorem, we improve this lower bound to $\Omega(\frac{2^d \log d}{\sqrt{d}})$ $n\choose 2d$.
\vspace{1mm}

\noindent \textbf{Keywords:} Geometric Hypergraph; Crossing Simplices; Ham Sandwich theorem; Gale Transform 
\end{abstract}

\section{Introduction}

Hypergraphs are a natural generalization of graphs. A hypergraph is a pair $(V,E)$ where $V$ is a finite set and $E \subseteq 2^{V}$ is a collection of subsets of $V$ \cite{matousek}. The elements of $E$ are called hyperedges. Given $n$ points in general position in $\mathbb{R}^d$, a geometric $(i+1)$-uniform hypergraph is defined as a collection of $i$-dimensional simplices as hyperedges, induced by some $(i + 1)$-tuples from the point set \cite{pach}. In this paper, we consider $i = d-1$. A complete geometric $d$-uniform hypergraph on $n$ vertices is represented as $K^d_n$ in this paper. The case $d=2, i=1$ has been studied in detail in the literature. The \textit{crossing number} of a geometric graph, known as the \textit{rectilinear crossing number}, is the minimum number of pairwise crossing edges in any of its straight-line drawings in $\mathbb{R}^2$, such that no three of its vertices lie on the same straight line. 

We define the {\it crossing number} $Cr_d(H)$ of a geometric hypergraph $H$ embedded in $\mathbb{R}^d$, for some $d \geq 2$, as the minimum possible number of \textit{pairwise crossings} of its hyperedges. A pair of hyperedges {\it overlap} if they have a common point in the relative interior of the simplices corresponding to them. It can be easily seen that a pair of $2$-simplices in $\mathbb{R}^3$ can overlap in two different ways. The first way, as shown in Figure \ref{fig1}, is called a \textit{crossing} and the second way is called an \textit{intersection}. Similarly in $\mathbb{R}^d$, there are various ways in which a pair of hyperedges can overlap. In order to define the crossing number $Cr_d(H)$ of a geometric hypergraph $H$ embedded in $\mathbb{R}^d$, we only count the crossing pair of hyperedges, i.e., a pair of overlapping hyperedges that have no vertices in common.  

As defined earlier, $Cr_2(K^2_n)$ denotes the number of crossing pair of edges in a straight-line drawing of $K^2_n$. The best known lower bound on this number is $Cr_2(K^2_n) > (0.375+\epsilon) {n \choose 4}$, where 
$\epsilon \approx 10^{-5}$  \cite{lovasz}. It is quite easy to show that the minimum number of pairwise crossing $2$-simplices in a complete geometric $3$-uniform hypergraph $K^3_n$ embedded in $\mathbb{R}^3$ is ${n \choose 6}$. This follows from the fact that any set of $6$ vertices in a general position in $\mathbb{R}^3$ contains a pair of crossing hyperedges. (See the Geometric Van Kampen-Flores Theorem below.) For a general dimension $d \geq 3$, let us denote by $c_d$ the minimum number of crossing pair of $(d-1)$-simplices spanned by a set of $2d$ vertices placed in general position in $\mathbb{R}^d$. This implies $Cr_d(K^d_n) \geq  c_d {n \choose 2d}$. 

In order to obtain a lower bound on $c_d$, we first use the geometric version of Van Kampen-Flores Theorem \cite{flores, kampen}. 

\begin{theorem}
(Geometric Van Kampen-Flores Theorem) For any $k \geq 1$, any set of $2k + 3$ points in $\mathbb{R}^{2k}$ contains two disjoint subsets $A$ and $B$ such that the convex hulls of $A$ and $B$ have a point common in their relative interior, and $|A| = |B| = k + 1$. 
\end{theorem}

If $d$ is even, i.e., $d = 2k$ for some $k$, this Theorem shows the existence of a crossing pair of 
$\frac{d}{2}$-simplices spanned by any set of $d+3$ vertices selected out of $2d$ vertices placed in general position in $\mathbb{R}^d$. This crossing pair can be extended to crossing pairs of $(d-1)$-simplices in $\binom{d-2}{\frac{d-2}{2}} = \Theta(\frac{2^d}{\sqrt{d}})$ ways. If $d$ is odd, i.e., $d = 2k'-1$ for some $k'$, we map all these $d+3$ vertices in $\mathbb{R}^{d}$ to $d+3$ vertices in $\mathbb{R}^{d+1}$ by adding a $0$ as the last coordinate of all these vertices. We also add one dummy vertex whose first $d$ coordinates are $0$ each, and whose last coordinate is non-zero. By the Geometric Van Kampen-Flores Theorem, this set of $2k'+3$ vertices in $\mathbb{R}^{2k'}$ contains a crossing pair of $(\frac{d+1}{2})$-simplices. Note that neither of these simplices contains the dummy vertex, as it is the only vertex in the $(d+1)$-st dimension and therefore can't be involved in a crossing. Therefore, this crossing pair can be extended to crossing pairs of $(d-1)$-simplices in $\binom{d-3}{\frac{d-3}{2}} = \Theta(\frac{2^d}{\sqrt{d}})$ ways.

\begin{figure}[t]
\label{fig1}
\centerline{
{\resizebox{4in}{1.5in}{\includegraphics{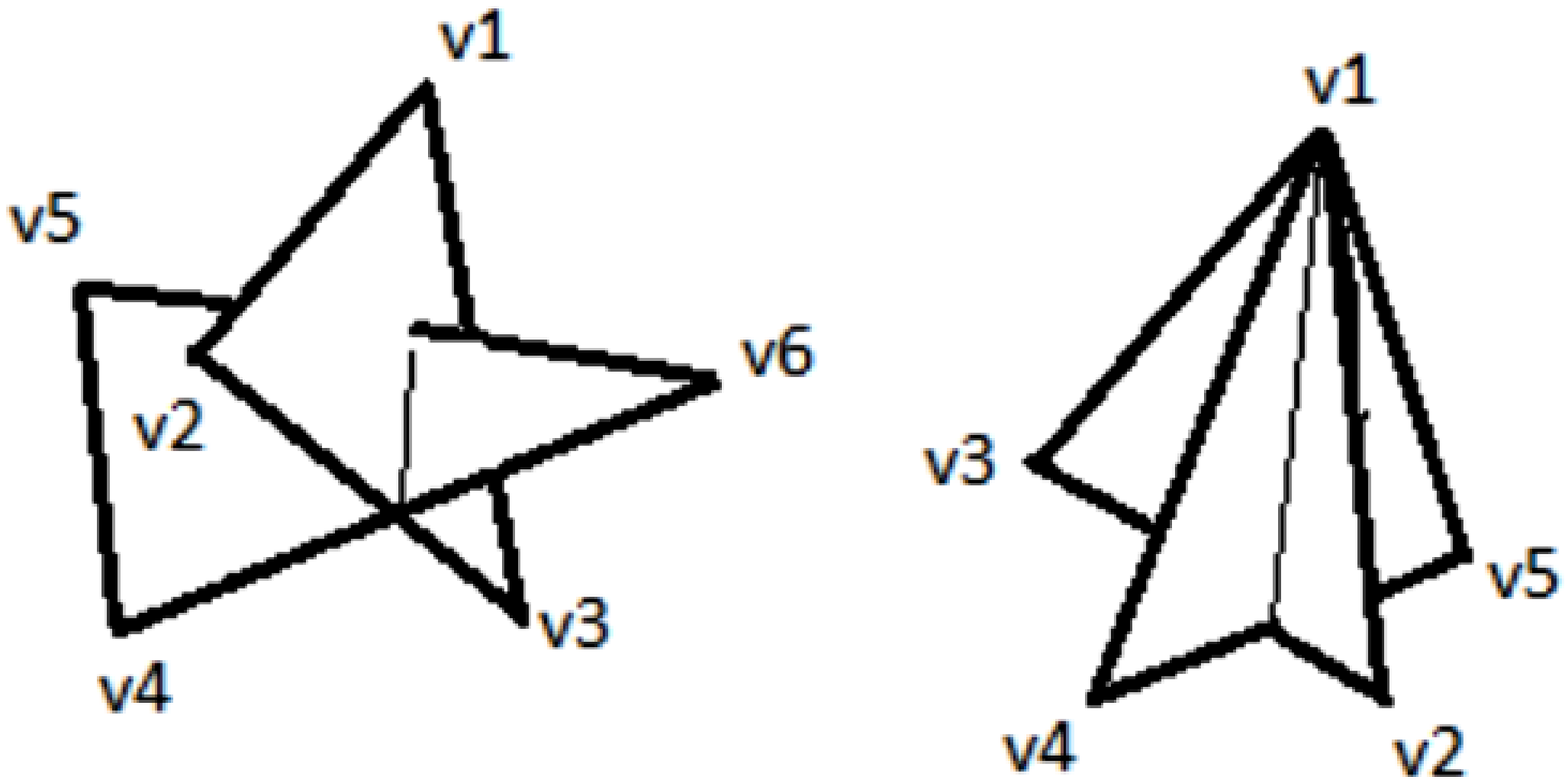}}}}
\caption{(i) (Left) Crossing between hyperedges $\{v1,v2,v3\}$ and $\{v4,v5,v5\}$. (ii) (Right) Intersection between hyperedges $\{v1,v2,v3\}$ and $\{v1,v4,v5\}$.}
\end{figure}

\subsection{Our Contribution}
In Section $3$, we first show that there are at least $4$ crossing pairs of $3$-simplices in a given set of $8$ points in general position in $\mathbb{R}^4$. This implies that $Cr_4(K^4_n) \geq  4 {n \choose 8}$. 
Thereafter, we use a similar idea to prove that $c_d = \Omega(\frac{2^d \log d}{\sqrt{d}})$. This implies that $Cr_d(K^d_n) = \Omega(\frac{2^d \log d}{\sqrt{d}})$ $n\choose 2d$.

As far as we know, this is the first non-trivial lower bound obtained on $Cr_d(K^d_{n})$. It is an exciting open problem to find out whether this lower bound is tight.

\section{Techniques used} 
\subsection{Gale Transformation}

For a positive integer $m$, consider a set $A$ of $m+d+1$ points $x_1,x_2\ldots x_{m+d+1}$ in general position in $\mathbb{R}^d$. It is easy to see that the matrix  
\[ M(x_1,x_2\ldots x_{m+d+1})= \left( \begin{array}{ccccc}x_1^{1} & x_2^{1} & ... & x_{m+d+1}^{1} \\ x_1^{2} & x_2^{2} & ... & x_{m+d+1}^{2} \\ ... & ... & ... & ... \\ x_1^{d} & x_2^{d} & ... & x_{m+d+1}^{d} \\ 1 & 1 & ... & 1 \end{array} \right)\] 
has $m$ null vectors. Let these null vectors be $(a^1_1, a^1_2, \ldots, a^1_{m+d+1})$,
$\ldots$, $(a^m_1, a^m_2, \ldots, a^m_{m+d+1})$. Consider the 
$m$-dimensional vectors $(a^1_1, a^2_1, \ldots, a^m_1)$, $(a^1_2, a^2_2, \ldots, a^m_2)$, $\ldots$, 
$(a^1_{m+d+1}, a^2_{m+d+1}, \ldots, a^m_{m+d+1})$. These $m+d+1$ points (representing $m+d+1$ vectors) in $\mathbb{R}^m$ are the Gale transform of $x_1, x_2, \ldots, x_{m+d+1}$ in $\mathbb{R}^d$. 
We denote the Gale transform of $A$ as $D(A)$. We define a \textit{proper linear separation} of the point set 
$D(A)$ to be a partition of $D(A)$ into two subsets $D_1(A)$ and $D_2(A)$ of size 
$\lfloor \frac{m+d+1}{2} \rfloor$ and $\lceil \frac{m+d+1}{2} \rceil$ respectively, by a hyperplane passing through origin. Then, we have the following:

\begin{lemma} 
\label{gale}
\cite{matousek1}
There is a bijection between the crossing pairs of $\lfloor \frac{m+d-1}{2} \rfloor$ and $\lceil \frac{m+d-1}{2} \rceil$-simplices in $A$ and the proper linear separations in $D(A)$.
\end{lemma}

\begin{lemma} 
\label{gale1}
\cite{matousek1}
The $m+d+1$ points in $A$ are in a general position in $\mathbb{R}^d$ if and only if every set of $m$ vectors, corresponding to a set of $m$ points from $D(A)$, spans $\mathbb{R}^{m}$.
\end{lemma}

\subsection{Ham Sandwich Theorem}

\begin{lemma}
\label{hamsandwich}
\cite{matousek}  Let $C_1, C_2, \ldots, C_d \in \mathbb{R}^d$ be finite point sets. Then, there exists a hyperplane $h$ that simultaneously bisects $C_1, C_2, \ldots, C_d$, i.e., each of open half-spaces 
defined by $h$ has at most $\lfloor \frac{1}{2} |C_i| \rfloor$ points of $C_i$.
\end{lemma}

\section{A Lower Bound on $Cr_4(K^4_8)$ and $Cr_d(K^d_{2d})$}
\label{4dimension}

\begin{lemma}
\label{8crossings}
There are at least $4$ crossing pairs of $3$-simplices spanned by a set of $8$ points in general position in 
$\mathbb{R}^4$.
\end{lemma}

\begin{proof}
Given a set $A$ of $8$ points in general position in $\mathbb{R}^4$, we consider the set $D(A)$ (the Gale transform) in $\mathbb{R}^3$ and obtain a lower bound on the number of proper linear separations of $D(A)$. By Lemma \ref{gale}, this lower bound implies a lower bound on the number of crossing pairs of $3$-simplices spanned by $A$. The Lemma \ref{gale1} implies that the vectors in $D(A)$ are in a general position, i.e., any set of $3$ vectors spans $\mathbb{R}^3$. Therefore, any hyperplane $h$ that passes through the origin would have at most two points from $D(A)$ on it. Any such hyperplane can be rotated so that we can make either of the two points go above or below $h$, while keeping the origin on it and maintaining the partitioning of the remaining points with respect to $h$. 

To apply the Ham Sandwich theorem in $\mathbb{R}^3$, we assign the origin to $C_3$ and create $2$ disjoint point sets $C_1$ and $C_2$ from the $8$ points in $D(A)$. We use the colors $c_1, c_2$ and $c_3$ to identify the sets $C_1, C_2$ and $C_3$ respectively. We use $c_3$ to color the origin. We proceed in the following manner with 
the colorings of the points in $D(A)$ with $c_1$ and $c_2$, assuming that $D(A) =\{p_1, p_2, p_3, p_4, p_5, p_6, p_7, p_8\}$.

\begin{itemize}
\item We color all the points with $c_1$, and don't color any of the points with $c_2$. It leads to a proper linear separation of $D(A)$ (through applying the Ham Sandwich theorem and rotating the partitioning hyperplane if required), which we assume to be $\{\{p_1, p_2, p_3, p_4\}$, $\{p_5, p_6, p_7, p_8\}\}$, without any loss of generality.
\item We color the points in first group of points $\{p_1, p_2, p_3, p_4\}$ with $c_1$ and the second group of points $\{p_5, p_6, p_7, p_8\}$ with $c_2$. It gives a new proper linear separation. Without any loss of generality, we can assume it to be $\{\{p_1, p_2, p_5, p_6\}$, $\{p_3, p_4, p_7, p_8\}\}$. Note that the pairs $\{p_1, p_2\}, \{p_3, p_4\}, \{p_5, p_6\}$ and  
$\{p_7, p_8\}$ have stayed together in both these separations.
\item We color the points $p_1$ and $p_2$ with $c_1$ and the rest of the points with $c_2$. In the resulting proper linear separation, the points $p_1$ and $p_2$ get separated and hence this proper linear separation is a new one. Here, we have two cases: \textbf{(i)} all of the remaining pairs, i.e., $\{p_3, p_4\}, \{p_5, p_6\}$ and $\{p_7, p_8\}$, get separated, \textbf{(ii)} one of these remaining pairs gets separated and the rest two pairs are still together.
\item In the \textbf{(ii)}-nd case, we just color the two points in one of the {\it unseparated} pairs with $c_1$ and rest all with $c_2$. We get a new proper linear separation.
\item In the \textbf{(i)}-st case, we can assume without a loss of generality that the proper linear separation is:  $\{p_1, p_3, p_5, p_7\}$, $\{p_2, p_4, p_6, p_8\}$. Consider the set of points $\{p_1, p_2, p_3, p_5\}$. In every proper linear separation obtained till now, three out of these four points have always been in the same partition. We color these points with $c_1$ and the rest with $c_2$. This gives a new proper linear separation. \qed
\end{itemize} 
\end{proof}

\noindent A similar argument can be used to prove the following result.

\begin{lemma}
\label{logdcrossings}
There are $\Omega(\frac{2^d \log d}{\sqrt{d}})$ crossing pairs of $(d-1)$-simplices spanned by a set of $2d$ points in general position in $\mathbb{R}^d$.
\end{lemma}

\begin{proof}
Consider any set $A'$ of $d+4$ points from the given set of $2d$ points in general position in $\mathbb{R}^d$. The Gale transform $D(A')$ of these $d+4$ points in $\mathbb{R}^d$ is a set of $d+4$ vectors in $\mathbb{R}^3$, such that any set of $3$ of these vectors spans $\mathbb{R}^{3}$. Therefore, any hyperplane $h$ that passes through the origin would have at most two points from $D(A')$ on it. Any such hyperplane can be rotated so that we can make either of the two points go above or below $h$, while keeping the origin on it and maintaining the partitioning of the remaining points with respect to $h$. 

We proceed as in Lemma \ref{8crossings}, i.e., we assign the origin to $C_3$ and create $2$ disjoint point sets $C_1$ and $C_2$ from the $d+4$ points in $D(A')$. Through applying the Ham Sandwich theorem and rotating the partitioning hyperplane if required, we get $\Theta(\log d)$ proper linear separations of $D(A')$. Each of these proper linear separations corresponds to a  $\lfloor \frac{d+2}{2} \rfloor$-simplex crossing a $\lceil \frac{d+2}{2} \rceil$-simplex spanned by $A$. Each of these crossing pairs can be extended to a crossing pair of $(d-1)$-simplices in $\binom{d-4}{\lceil \frac{d-4}{2} \rceil} = \Theta(\frac{2^d}{\sqrt{d}})$ ways.
\qed
\end{proof}

\end{document}